\documentclass[a4 paper,12pt,bezier]{article}
\usepackage[top=3cm,bottom=3cm,left=2.5cm,right=2.5cm]{geometry}
\usepackage{amsmath}
\usepackage{amsfonts,amsthm,amssymb}
\usepackage{amsfonts}
\usepackage{graphics,subfigure,caption2}
\usepackage{tikz, color}
\usepackage{graphics,epstopdf}
\usepackage{latexsym,bm}
\usepackage{amsfonts,amsthm,amssymb,mathrsfs,bbding}
\usepackage{hyperref}
\usepackage{CJK}
\usepackage{color}
\usepackage{indentfirst}
\usepackage{graphicx}
\usepackage{cleveref}
\usepackage{booktabs}
\usepackage{multirow}
\usepackage{cleveref}

\newtheorem{lem}{Lemma}[section]
\newtheorem{thm}[lem]{Theorem}

\newtheorem{con}[lem]{Conjecture}

\newtheorem{prob}[lem]{Problem}

\theoremstyle{plain}

\usepackage{cite}

\begin{document}
	\begin{CJK}{GBK}{song}
		\title{Spectral and size conditions for spanning $k$-trees in tough graphs}
		\author{Siyuan Liang, Tao Tian\footnote{Corresponding author. E-mail: 13926736123@163.com (S. Liang), taotian0118@163.com (T. Tian).}\\
			\small  School of Mathematics and Statistics, Key Laboratory of Analytical Mathematics and\\
			\small  Applications $($Ministry of Education$)$, Fujian Key Laboratory of Analytical Mathematics and\\
			\small  Applications (FJKLAMA), Center for Applied Mathematics of Fujian Province $($FJNU$)$,\\
			\small Fujian Normal University,
			\small Fuzhou 350117, PR China.}
		\date{}
		\maketitle

		\maketitle {\flushleft\bf Abstract}:
		The toughness of a graph is a crucial parameter for characterizing its structural properties. The toughness of a non-complete graph $G$ is defined as $\tau(G) = \min \{ \dfrac{|S|}{c(G - S)} : S \subseteq V(G), c(G-S) > 1 \}$, where $c(G)$ denotes the number of components of $G$. We define $\tau(K_n) = \infty$. A graph $G$ is said to be $\tau$-tough if $|S| \ge \tau \cdot c(G-S)$ for every vertex cut $S$ of $G$. Let $k \ge 3$ be an integer. For $\frac{1}{k-\eta}$-tough graphs with $\eta \in \{0, 1\}$, Liu, Fan and Shu \cite{a34} derived sufficient conditions in terms of the spectral radius and the signless Laplacian spectral radius for the existence of a spanning $k$-tree. Jia and Lu \cite{a24}, for the case $\frac{1}{k-1} \leq \tau(G) < \frac{1}{k-2}$, established sufficient conditions in terms of the spectral radius and the signless Laplacian spectral radius for the existence of a spanning $k$-tree. Motivated by these results, in this paper, we further investigate sufficient conditions for the existence of a spanning $k$-tree when $\frac{1}{k} \leq \tau(G) < \frac{1}{k-1}$. Specifically, for a connected $\frac{t}{t(k-1)+1}$-tough graph of sufficiently large order $n$ (where $t \ge 1$ is an integer), we provide sufficient conditions for the existence of a spanning $k$-tree in terms of the spectral radius and the signless Laplacian spectral radius. Furthermore, we establish a lower bound on the size (number of edges) to guarantee the existence of a spanning $k$-tree.

		\maketitle {\flushleft\textit{\bf Keywords}}: Toughness; Spanning $k$-tree; Spectral radius; Signless Laplacian spectral radius

		\section{  Introduction  }\label{sec1}
		This paper considers only finite, undirected, simple graphs. Let $G=(V(G), E(G))$ be a graph, where $V(G)$ is the vertex set and $E(G)$ is the edge set. The order and size of graph $G$ are denoted by $|V(G)|=n$ and $|E(G)|=e(G)$, respectively. For $v\in V(G)$, we let $N_G(v)$ and $d_G(v)$ denote the \emph{neighborhood} and the \emph{degree} of $v$ in $G$, respectively. The maximum degree of a graph $G$ is denoted by $\Delta(G)$. The number of components of graph $G$ is denoted by $c(G)$. Let $G_1$ and $G_2$ be two disjoint graphs. The union $G_1 \cup G_2$ is the graph with vertex set $V(G_1) \cup V(G_2)$ and edge set $E(G_1) \cup E(G_2)$. The join $G_1 \vee G_2$ is defined as the graph obtained from $G_1 \cup G_2$ by adding edges connecting every vertex of $G_1$ to every vertex of $G_2$. For any real number $c$, let $\lceil c \rceil$ denote the smallest integer greater than or equal to $c$.
		
		Let $G$ be a graph with vertex set $\{v_1, v_2, \dots, v_n\}$. Its adjacency matrix is defined as $A(G)=(a_{ij})_{n \times n}$, where $a_{ij}=1$ if $v_i v_j \in E(G)$ and $a_{ij}=0$ otherwise. The degree diagonal matrix of a graph $G$, denoted by $D(G)$, is the diagonal matrix whose diagonal entries are the degrees of the vertices of $G$. The signless Laplacian matrix of graph $G$ is defined as $Q(G)=D(G)+A(G)$. The eigenvalues of $A(G)$ and $Q(G)$ are called the adjacency eigenvalues and signless Laplacian eigenvalues of $G$, respectively. The largest eigenvalues of $A(G)$ and $Q(G)$ are called the  spectral radius and the signless Laplacian spectral radius of $G$, respectively, and are denoted by $\rho(G)$ and $q(G)$, respectively.
		
		 A cycle (resp. path) in $G$ is called a Hamilton cycle (resp. Hamilton path) if it contains every vertex of $G$. A graph is said to be Hamiltonian if it contains a Hamilton cycle. Let $k \geq 2$ be an integer. A tree $T$ is called a \textit{k-tree} if $d_T(v) \leq k$ for every vertex $v \in V(T)$. A \textit{k-tree} $T$ is called a \textit{spanning k-tree} of a connected graph $G$ if $V(T) = V(G)$. Obviously, a Hamilton path is a spanning $2$-tree of a graph.
		
		A graph $G$ is said to be \textit{$\tau$-tough} if $|S| \geq \tau \cdot c(G - S)$ for every vertex subset $S \subseteq V(G)$ with $c(G - S) \geq 2$. The \textit{toughness} $\tau(G)$ of a graph $G$ is the maximum value of $\tau$ for which $G$ is $\tau$-tough (by convention, the toughness of a complete graph $K_n$ is $\tau(K_n) = \infty$). Therefore, if $G$ is not a complete graph, then
		\[
		\tau(G) = \min\left\{ \frac{|S|}{c(G - S)} : S \subseteq V(G), c(G-S) > 1 \right\}.
		\]
		 It is easy to verify that every Hamiltonian graph is $1$-tough, but the converse is not true. The following conjecture on Hamiltonicity, proposed by Chv\'{a}tal~\cite{a25}, remains open.

		\begin{con}[Chv\'{a}tal \cite{a25}]\label{conjecture 1.1}
			There exists a constant $t_0$ such that every $t_0$-tough graph with at least $3$ vertices is Hamiltonian.
		\end{con}
		
		Bauer, Broersma and Veldman \cite{a20} constructed $(\frac{9}{4} - \varepsilon)$-tough non-Hamiltonian graphs for arbitrary $\varepsilon > 0$;  therefore, if Conjecture \ref{conjecture 1.1} holds, then $t_0 \geq \frac{9}{4}$. This conjecture has been verified for several well-studied classes of graphs in numerous papers, see \cite{a21, a22, a23, a33, a5, a6}.

		In 1989, Win \cite{a2} gave a sufficient toughness condition for the existence of a spanning $k$-tree in connected graphs. Numerous researchers have established a variety of spectral sufficient conditions for the existence of spanning $k$-trees in graphs. Ning and Ge \cite{a31} investigated spectral radius conditions that guarantee the existence of a spanning 2-tree and a Hamilton cycle in graphs. Liu, Shiu and Xue \cite{a14} provided sufficient signless Laplacian spectral radius and size conditions for the existence of a spanning 2-tree and a Hamilton cycle in bipartite graphs. For any integer $k \geq 3$, Fan et al. \cite{a37} established sufficient spectral radius and signless Laplacian spectral radius conditions for connected graphs that guarantee the existence of a spanning $k$-tree. For $k \geq 4$, Zhou and Wu \cite{a4} derived distance spectral radius conditions that ensure the existence of a spanning $k$-tree in connected graphs. For $k \geq 4$, Zhou, Zhang and Liu \cite{a28} obtained an upper bound on the distance signless Laplacian spectral radius that guarantees the existence of a spanning $k$-tree.
		
		Wu \cite{a26} established sharp lower bounds on the number of edges and the spectral radius for a connected graph to guarantee the existence of a spanning tree with leaf degree at most $k$. Chen et al. \cite{a30} derived sufficient conditions on the number of edges, the spectral radius, and the signless Laplacian spectral radius that guarantee the existence of a spanning tree with leaf distance at least 4. Ao, Liu and Yuan \cite{a27} provided sufficient conditions, in terms of the number of $r$-cliques, for a graph with a specific minimum degree to be $k$-factor-critical, and established bounds on various spectral parameters that guarantee the existence of a spanning $k$-tree. Chen, Li and Xu \cite{a39} derived sufficient conditions based on the spectral radius of the matrix $A_\alpha(G) = \alpha D(G) + (1 - \alpha)A(G)$, algebraic connectivity, nullity, and energy for determining whether a graph contains a spanning $k$-ended tree.

		In \cite{a34}, Liu, Fan and Shu presented the following spectral conditions that guarantee the existence of a spanning $k$-tree.
		
		\begin{thm}[\textnormal{Liu, Fan and Shu}\cite{a34}]
			Let $G$ be a connected $\frac{1}{k-\eta}$-tough graph of order $n$ with $k \ge 3$ and $\eta \in \{0, 1\}$. Each of the following holds.
			
			(i) If $n \ge 8k + 12$ and $\rho(G) \ge \rho(K_{\eta+2} \vee (K_{n-(k-1)(\eta+2)-2} \cup [(k - 2)(\eta + 2) + 2]K_1))$, then $G$ contains a spanning $k$-tree unless $G \cong K_{\eta+2} \vee (K_{n-(k-1)(\eta+2)-2} \cup [(k - 2)(\eta + 2) + 2]K_1)$.
			
			(ii) If $n \ge 11k + 47$ and $q(G) \ge q(K_{\eta+2} \vee (K_{n-(k-1)(\eta+2)-2} \cup [(k - 2)(\eta + 2) + 2]K_1))$, then $G$ contains a spanning $k$-tree unless $G \cong K_{\eta+2} \vee (K_{n-(k-1)(\eta+2)-2} \cup [(k - 2)(\eta + 2) + 2]K_1)$.
		\end{thm}
		
		Recently, for the case $\frac{1}{k-1} \leq \tau(G) < \frac{1}{k-2}$, Jia and Lu \cite{a24} established sufficient conditions for the existence of a spanning $k$-tree in connected $\frac{t}{t(k-2)+1}$-tough graphs in terms of the spectral radius and the signless Laplacian spectral radius.
		
		For the case $\frac{1}{k} \le \tau (G) < \frac{1}{k-1}$, the following problem was proposed by Jia and Lu in \cite{a24}.
		
		\begin{prob}[Jia and Lu \cite{a24}]\label{problem 1.3}
			When $\frac{1}{k} \le \tau < \frac{1}{k-1}$, what spectral conditions can guarantee the existence of a spanning $k$-tree in a $\tau$-tough graph?
		\end{prob}
		
		In this paper, we provide an answer to Problem \ref{problem 1.3}. In the following, let $N_1(s_0, k) = \max \Big\{ (s_0+1)k + s_0 + 3 + \frac{2s_0+5}{k-2} + \frac{4}{(k-2)^2}, \allowbreak s_0^2 + \frac{2k^2-5k+6}{2(k-2)}s_0 + \frac{k^2+3k-8}{2(k-2)} \Big\}$, and let $N_2(s_0, k) = \max \Big\{ (s_0+1)k + s_0 + 3 + \frac{2s_0+5}{k-2} + \frac{4}{(k-2)^2}, \allowbreak \frac{1}{2}(s_0+1)^2k + \frac{3s_0+5}{2} + \frac{2s_0+3}{k-2} \Big\}$, where $s_0 = \left\lceil \frac{3t}{t+1} \right\rceil$, $k \ge 3$ and $t \ge 1$ are integers.
		
		\begin{thm}\label{Theorem 1.3}
			Let $G$ be a connected $\frac{t}{t(k-1)+1}$-tough graph of order $n$, and let $s_0 = \left\lceil \frac{3t}{t+1} \right\rceil$, where $k \ge 3$ and $t \ge 1$ are integers.  Each of the following holds.
			
			(i) If $n \ge N_1(s_0, k) $ and $\rho(G) \geq \rho(K_{s_0} \vee (K_{n-s_0(k-1)-2} \cup (s_0(k-2)+2)K_1))$, then $G$ contains a spanning $k$-tree.
			
			(ii) If $n \ge N_2(s_0, k) $ and $	q(G) \geq q(K_{s_0} \vee (K_{n-s_0(k-1)-2} \cup (s_0(k-2)+2)K_1))$, then $G$ contains a spanning $k$-tree.
		\end{thm}
		
		Next, we establish a lower bound on the size of a graph that guarantees the existence of spanning $k$-tree.
		
		\begin{thm}\label{Theorem 1.4}
			Let $G$ be a connected $\frac{t}{t(k-1)+1}$-tough graph with $n$ vertices, and let $s_0 = \left\lceil \frac{3t}{t+1} \right\rceil$, where $k \ge 3$ and $t \ge 1$ are integers. If $n \ge \frac{s_0k^3 - (3s_0-2)k^2 + (2s_0-5)k + 6}{(k-2)^2}$, and
			\[
			e(G) > \binom{n-s_0(k-2)-2}{2} + s_0(s_0(k-2)+2),
			\]
			then $G$ contains a spanning $k$-tree.
		\end{thm}

 	   The rest of this paper is organized as follows. In Section \ref{sec2}, we introduce some preliminary lemmas concerning spectral radius $\rho(G)$ and signless Laplacian spectral radius $q(G)$ of a graph $G$, and present some lemmas concerning the structural properties of graphs. In Section \ref{sec3}, the proofs of our main results are given.

		\section{Preliminaries}\label{sec2}
		
		Before proceeding to the proofs of our main results, we introduce several useful lemmas.
		
		\begin{lem}[\textnormal{Win}\cite{a2}]\label{Lemma 2.1}
			Let $k \geq 3$ be an integer. If a connected graph $G$ satisfies that for every vertex subset $S \subseteq V(G)$,
			\[
			c(G - S) \leq (k - 2)|S| + 2,
			\]
			then $G$ contains a spanning $k$-tree.
		\end{lem}

		\begin{lem}[\textnormal{Li and Feng}\cite{a3}]\label{Lemma 2.2}
			Let $G$ be a connected graph and $H$ be a subgraph of $G$. Then
			\[
			\rho(G) \geq \rho(H),
			\]
			with equality if and only if $G = H$.
		\end{lem}
		
		
		\begin{lem}[\textnormal{Fan et al.}\cite{a37}]\label{Lemma 2.3}
		Let $\lambda(G) \in \{\rho(G), q(G)\}$, and let $n = \sum_{i=1}^{t} n_i + s$. If $n_1 \ge n_2 \ge \dots \ge n_t \ge 1$ and $n_1 < n - s - t + 1$, then
		\[
		\lambda(K_s \vee (K_{n_1} \cup K_{n_2} \cup \dots \cup K_{n_t})) < \lambda(K_s \vee (K_{n-s-t+1} \cup (t-1)K_1)).
		\]
		\end{lem}
		
		\begin{lem}[\textnormal{Hong}\cite{a36}]\label{Lemma 2.4}
		Let $G$ be a graph on $n$ vertices. Then
		\[
		\rho(G) \leq \sqrt{2e(G) - n + 1},
		\]
		with equality if and only if $G$ is a star graph or a complete graph.
		\end{lem}
		
		\begin{lem}[\textnormal{Shen et al.}\cite{a8}]\label{Lemma 2.5}
		 Let $G$ be a connected graph, and let $H$ be a subgraph of $G$. Then
		\[
		q(G) \ge q(H),
		\]
		with equality if and only if $G = H$.
		\end{lem}
		
		\begin{lem}[\textnormal{Das}\cite{a40}]\label{Lemma 2.6}
			 Let $G$ be a graph with $n$ vertices. Then
			\[
			q(G) \le \frac{2e(G)}{n - 1} + n - 2,
			\]
			with equality if and only if $G$ is a star graph or a complete graph.
		\end{lem}
		
		\begin{lem}[\textnormal{Chen, Li and Xu}\cite{a1}]\label{Lemma 2.7}
			 Let $n = \sum_{i=1}^t n_i + s$. If $n_1 \ge n_2 \ge \dots \ge n_t \ge 1$ and $n_1 \le n - s - t + 1$, then
			\[
			e(K_s \vee (K_{n_1} \cup K_{n_2} \cup \dots \cup K_{n_t})) \le e(K_s \vee (K_{n-s-t+1} \cup (t-1)K_1)).
			\]
		\end{lem}
		
		\begin{lem}\label{Lemma 3.1}
			Let $G$ be a connected $\frac{t}{t(k-1)+1}$-tough graph, where $t \ge 1$ and $k \ge 3$ are integers. If $G$ contains no spanning $k$-tree, then there exists a non-empty subset $S \subseteq V(G)$ such that $c(G - S) \ge (k - 2)|S| + 3$ and $|S| \ge \left\lceil \frac{3t}{t + 1} \right\rceil$.
		\end{lem}
		
		\begin{proof}
			Suppose $G$ contains no spanning $k$-tree. By Lemma \ref{Lemma 2.1}, there exists a non-empty subset $S$ of $V(G)$ such that
			\[
			c(G-S) \ge (k-2)|S| + 3.
			\]
			
			By the definition of a $\frac{t}{t(k-1)+1}$-tough graph,
			\[
			\frac{|S|}{c(G-S)} \ge \tau(G) \ge \frac{t}{t(k-1)+1}.
			\]
			
			Then, $(t(k-1)+1)|S| \ge t \cdot c(G-S) \ge t((k-2)|S|+3)$.
			
			Thus, $|S| \ge \left\lceil \frac{3t}{t+1} \right\rceil$.
		\end{proof}

		\section{ Proofs of main results }\label{sec3}
		
		In this section, the proofs of our main results are given.

		\subsection{Proof of Theorem \ref{Theorem 1.3}}
		
		\begin{proof}[\textbf{Proof of Theorem \ref{Theorem 1.3}}]
		 Suppose that the connected $\frac{t}{t(k-1)+1}$-tough graph $G$ contains no spanning $k$-tree, where $t \ge 1$ and $k \ge 3$ are integers. Let $|S|=s$ and $s_0 = \left\lceil \frac{3t}{t+1} \right\rceil$. By Lemma \ref{Lemma 3.1}, $s \ge s_0$, and there exists a graph $G_1 = K_s \vee (K_{n_1} \cup K_{n_2} \cup \dots \cup K_{n_{(k-2)s+3}})$ such that $G$ is a spanning subgraph of $G_1$, where $n_1 \ge n_2 \ge \dots \ge n_{(k-2)s+3} \ge 1$ and $\sum_{i=1}^{(k-2)s+3} n_i = n-s$.
		
		Let $\lambda(G) \in \{\rho(G), q(G)\}$. By Lemmas \ref{Lemma 2.2} and   \ref{Lemma 2.5}, $\lambda(G) \le \lambda(G_1)$, where equality holds if and only if $G \cong G_1$. Let $G_2 = K_s \vee (K_{n-(k-1)s-2} \cup ((k-2)s+2)K_1)$, where $n \ge (k-1)s+3$. Since $n_1 = n - s - \sum_{i=2}^{(k-2)s+3} n_i \le n - s - [(k-2)s+2] = n - s - [(k-2)s+3]+1$, by Lemma \ref{Lemma 2.3}, $\lambda(G_1) \le \lambda(G_2)$, where equality holds if and only if $(n_1, n_2, \dots, n_{(k-2)s+3}) = (n - (k-1)s - 2, 1, \dots, 1)$.
		
		If $s = s_0$, then $G_2 = K_{s_0} \vee (K_{n-s_0(k-1)-2} \cup (s_0(k-2)+2)K_1)$. Therefore,
		\[\lambda(G) \le \lambda (K_{s_0} \vee (K_{n-s_0(k-1)-2} \cup (s_0(k-2)+2)K_1)),\]
		where equality holds if and only if $G \cong K_{s_0} \vee (K_{n-s_0(k-1)-2} \cup (s_0(k-2)+2)K_1)$. By assumption, $G \cong K_{s_0} \vee (K_{n-s_0(k-1)-2} \cup (s_0(k-2)+2)K_1)$. Since $t \ge 1$, $s_0 = \lceil \frac{3t}{t+1} \rceil \in \{2, 3\}$. Since $n \geq N_i(s_0, k)$ ($i=1, 2$) and $k \geq 3$, $n \geq (s_0+1)k+s_0+3$. Then $n - (k - 1)s_0 - 2 \geq (s_0 + 1)k + s_0 + 3 - (k - 1)s_0 - 2 = k + 2s_0 + 1 \geq 8$. So, the complete graph $K_{n-s_0(k-1)-2}$ has Hamilton path $P$ of order at least 8. Thus, we can easily find a spanning $k$-tree in $G$ (depicted in Fig. \ref{fig 1}), contradicting the assumption that $G$ has no spanning $k$-tree.
		
		\begin{figure}[htbp]
			\centering
			\includegraphics[width=0.8\linewidth]{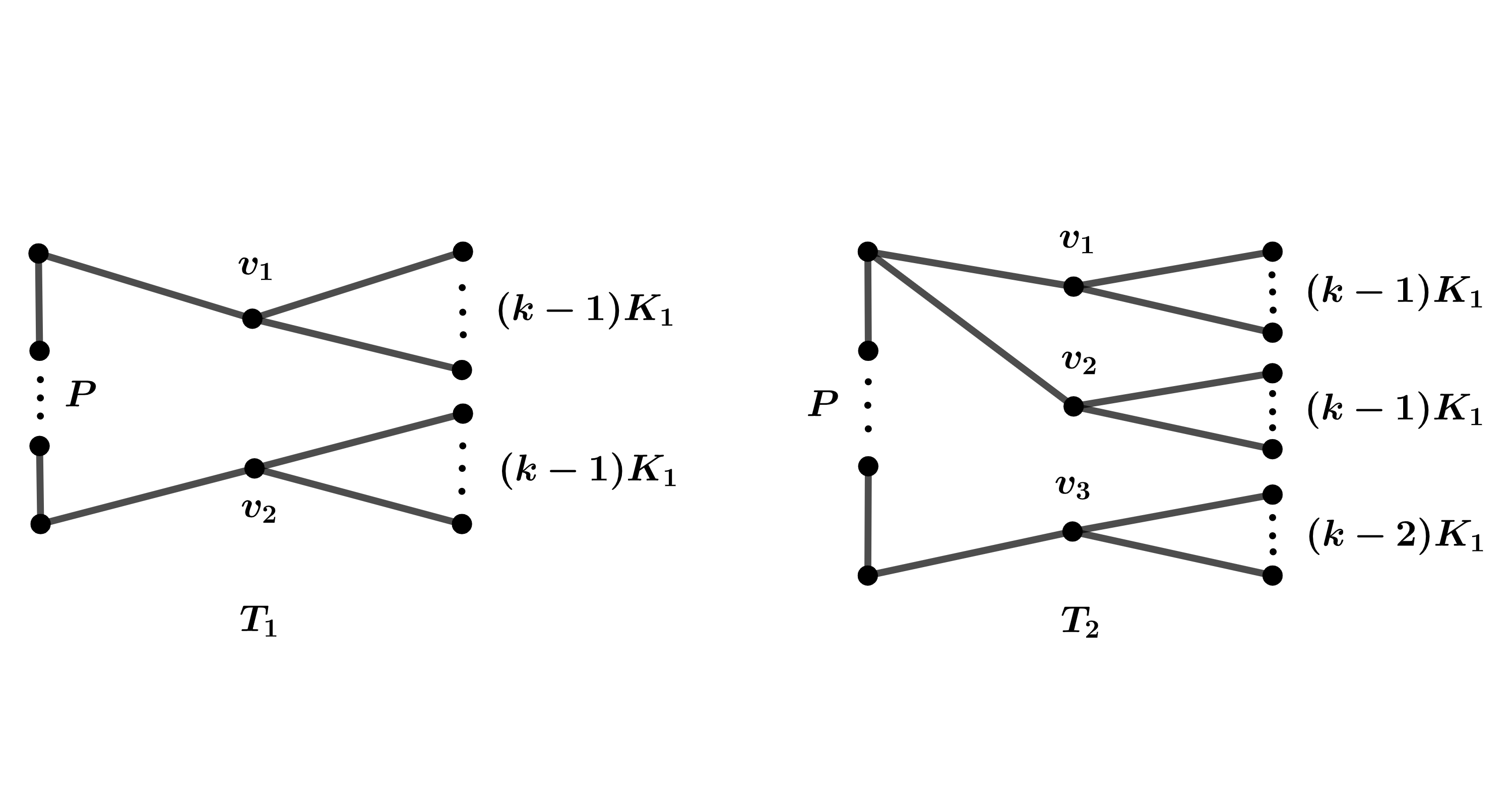}
			\caption{A spanning $k$-tree $T_1$ for $s_0 = 2$ and a spanning $k$-tree $T_2$ for $s_0 = 3$.}
			\label{fig 1}
		\end{figure}

		Hence $s \ge s_0 + 1$. Then
		\begin{align}
			2e(G_2) &= 2 \binom{n-(k-2)s-2}{2} + 2s((k-2)s+2) \notag \\
			&= k(k-2)s^2 + (-2kn+4n+5k-6)s + n^2 - 5n + 6 \label{1}.
		\end{align}
		
		 Next we distinguish the following two cases.
		
		\textbf{Case \textcolor{teal}{1}.} $\lambda(G) = \rho(G)$.

		By \eqref{1} and Lemma \ref{Lemma 2.4},
		\begin{align}
			\rho(G_2) &\le \sqrt{2e(G_2)-n+1} \notag \\
			&= \sqrt{k(k-2)s^2 + (-2kn+4n+5k-6)s + n^2 - 6n + 7}. \label{2}
		\end{align}
		
		Let $f(s) = k(k-2)s^2 + (-2kn+4n+5k-6)s + n^2 - 6n + 7$. Since $s \ge s_0+1$ and $n \ge (k-1)s+3$, $s_0+1 \le s \le \frac{n-3}{k-1}$. Since $t \ge 1$ and $k \ge 3$ are integers, and $n \ge N_1(s_0, k) $, a direct algebraic calculation yields
		\begin{align*}
	& f(s_0+1) - f\left(\frac{n-3}{k-1}\right) \\
	={}& \frac{(n-s_0k-k+s_0-2) [(k-2)^2n - (s_0+1)k^3 + (3s_0+1)k^2 - (2s_0-3)k - 6]}{(k-1)^2} \\
	\ge{}& 0.
		\end{align*}
		
		Note that $f(s)$ is a convex function of $s$, since the leading coefficient $k(k-2)$ is positive for $k \ge 3$. Consequently, $f(s)$ attains its maximum on the interval $[s_0 + 1, \frac{n-3}{k-1}]$ at one of its endpoints. Combining this with the inequality $f(s_0 + 1) \ge f\left(\frac{n-3}{k-1}\right)$ established above, we conclude that the maximum value of $f(s)$ is attained at $s = s_0 + 1$. By combining with the assumptions,  we obtain
		\begin{align*}
			\rho(G_2) &\le \sqrt{f(s_0+1)} \\
			&= \Big[ (n-s_0(k-2)-3)^2 - 2(k-2)n + 2(k-2)s_0^2 + (2k^2-5k+6)s_0 + k^2+3k-8 \Big]^{\frac{1}{2}} \\
			&\le \Big[ (n-s_0(k-2)-3)^2 - 2(k-2)\left( s_0^2 + \frac{2k^2-5k+6}{2(k-2)}s_0 + \frac{k^2+3k-8}{2(k-2)} \right) \\
			&\quad + 2(k-2)s_0^2 + (2k^2-5k+6)s_0 + k^2+3k-8 \Big]^{\frac{1}{2}} \\
			&= \sqrt{(n-s_0(k-2)-3)^2} \\
			&= n-s_0(k-2)-3.
		\end{align*}
		
		Since $K_{n-s_0(k-2)-2}$ is a proper subgraph of $K_{s_0} \vee (K_{n-s_0(k-1)-2} \cup (s_0(k-2)+2)K_1)$, by Lemma \ref{Lemma 2.2},
		\begin{align*}
			\rho(G) \le \rho(G_1) \le \rho(G_2) &\le n - s_0(k-2) - 3 = \rho(K_{n-s_0(k-2)-2}) \\
			&< \rho(K_{s_0} \vee (K_{n-s_0(k-1)-2} \cup (s_0(k-2)+2)K_1)),
		\end{align*}
		contradicting the assumption that $\rho(G) \ge \rho(K_{s_0} \vee (K_{n-s_0(k-1)-2} \cup (s_0(k-2)+2)K_1))$.

		\textbf{Case \textcolor{teal}{2}.} $\lambda(G) = q(G)$.
		
		By \eqref{1} and Lemma \ref{Lemma 2.6},
		\begin{align}
			q(G_2) &\le \frac{2e(G_2)}{n-1} + n - 2 \notag \\
			&= \frac{k(k-2)s^2 + (-2kn+4n+5k-6)s + 2n^2 - 8n + 8}{n-1} \label{3}.
		\end{align}
		
		Let $g(s) = k(k-2)s^2 + (-2kn+4n+5k-6)s + 2n^2 - 8n + 8$. Since $s \ge s_0+1$ and $n \ge (k-1)s+3$, $s_0+1 \le s \le \frac{n-3}{k-1}$. Since $t \ge 1$ and $k \ge 3$ are integers, and	$n \ge N_2(s_0, k) $, a direct algebraic calculation yields
		\begin{align*}
			& g(s_0+1) - g\left(\frac{n-3}{k-1}\right) \\
			={}& \frac{(n-s_0k+s_0-k-2) \left[ (k-2)^2n - (s_0+1)k^3 + (3s_0+1)k^2 - (2s_0-3)k - 6 \right]}{(k-1)^2} \\
			\ge{}& 0.
		\end{align*}
		
		Note that $g(s)$ is a convex function of $s$, since the leading coefficient $k(k-2)$ is positive for $k \ge 3$. Therefore, its maximum value on the closed interval $[s_0 + 1, \frac{n-3}{k-1}]$ must be attained at one of the endpoints. Combining this with the inequality $g(s_0 + 1) \ge g\left(\frac{n-3}{k-1}\right)$ established above, we conclude that $g(s) \le g(s_0 + 1)$. By combining with the assumptions, we obtain
		\begin{align*}
			q(G_2) &\le \frac{g(s_0+1)}{n-1} \\
			&= \frac{2n^2 - 2(s_0k - 2s_0 + k + 2)n + (s_0^2 + 2s_0 + 1)k^2 - (2s_0^2 - s_0 - 3)k - 6s_0 + 2}{n-1} \\
			&= 2(n - s_0(k-2) - 3) - \frac{2(k-2)n - (s_0+1)^2k^2 + (2s_0^2+s_0-3)k + 2s_0+4}{n-1} \\
			&\le 2(n - s_0(k-2) - 3) - \frac{1}{n - 1} \Bigg[ 2(k-2)\left( \frac{1}{2}(s_0 + 1)^2 k + \frac{3s_0 + 5}{2} + \frac{2s_0 + 3}{k - 2} \right) \\
			&\quad - (s_0 + 1)^2 k^2 + (2s_0^2 + s_0 - 3)k + 2s_0 + 4 \Bigg] \\
			&= 2(n - s_0(k-2) - 3).
		\end{align*}
		
		Since $K_{n-s_0(k-2)-2}$ is a proper subgraph of $K_{s_0} \vee (K_{n-s_0(k-1)-2} \cup (s_0(k-2)+2)K_1)$, by Lemma \ref{Lemma 2.5},
		\begin{align*}
			q(G) &\le q(G_1) \le q(G_2) \le 2(n-s_0(k-2)-3) = q(K_{n-s_0(k-2)-2}) \\
			&< q(K_{s_0} \vee (K_{n-s_0(k-1)-2} \cup (s_0(k-2)+2)K_1)),
		\end{align*}
		contradicting the assumption that $	q(G) \ge q(K_{s_0} \vee (K_{n-s_0(k-1)-2} \cup (s_0(k-2)+2)K_1))$.
		
		This completes the proof of Theorem \ref{Theorem 1.3}.
	\end{proof}

		\subsection{Proof of Theorem \ref{Theorem 1.4}}
		\begin{proof}[\textbf{Proof of Theorem \ref{Theorem 1.4}}]
		Suppose that the connected $\frac{t}{t(k-1)+1}$-tough graph $G$ contains no spanning $k$-tree, where $t \ge 1$ and $k \ge 3$ are integers. Let $|S|=s$ and $s_0 = \left\lceil \frac{3t}{t+1} \right\rceil$. By Lemma \ref{Lemma 3.1}, $s \ge s_0$, and there exists a graph $G_1 = K_s \vee (K_{n_1} \cup K_{n_2} \cup \dots \cup K_{n_{(k-2)s+3}})$ such that $G$ is a spanning subgraph of $G_1$, where $n_1 \ge n_2 \ge \dots \ge n_{(k-2)s+3} \ge 1$ and $\sum_{i=1}^{(k-2)s+3} n_i = n-s$. Then,
		\[
		 e(G) \le e(G_1).
		\]
		
		Let $G_2 = K_s \vee (K_{n-(k-1)s-2} \cup ((k-2)s+2)K_1)$, where $n \ge (k-1)s+3$. Since $n_1 = n - s - \sum_{i=2}^{(k-2)s+3} n_i \le n - s - [(k-2)s+2] = n - (k-1)s - 2$, by Lemma \ref{Lemma 2.7},
		\[
		e(G_1) \le e(G_2).
		\]
		
		Let $G_3 = K_{s_0} \vee (K_{n-(k-1)s_0-2} \cup ((k-2)s_0+2)K_1)$. Then,
		\[
		e(G_3) = \binom{n-(k-2)s_0-2}{2} + s_0((k-2)s_0+2).
		\]
		
		If $s = s_0$, then $G_2 = G_3$. Then,
		\[e(G) \le e(G_1) \le e(G_2) = e(G_3) = \binom{n-(k-2)s_0-2}{2} + s_0((k-2)s_0+2),\]
		contradicting the assumption that $	e(G) > \binom{n-s_0(k-2)-2}{2} + s_0(s_0(k-2)+2)$.
		
		Then $s \ge s_0+1$. Since $n \ge (k-1)s+3$, $s_0+1 \le s \le \frac{n-3}{k-1}$. Note that $n \ge \frac{s_0k^3 - (3s_0-2)k^2 + (2s_0-5)k + 6}{(k-2)^2}$. Then
		\begin{align*}
			& \binom{n-s_0(k-2)-2}{2} + s_0(s_0(k-2)+2) - e(G_2) \\
			={}& \binom{n-s_0(k-2)-2}{2} + s_0(s_0(k-2)+2) - \binom{n-(k-2)s-2}{2} - s(s(k-2)+2) \\
			={}& \frac{1}{2}(s-s_0)(2(k-2)n - (k^2-2k)s - s_0k^2 + (2s_0-5)k + 6) \\
			\ge{}& \frac{1}{2}(s-s_0)\left(2(k-2)n - (k^2-2k)\frac{n-3}{k-1} - s_0k^2 + (2s_0-5)k + 6\right) \\
			={}& \frac{s-s_0}{2(k-1)} \left((k-2)^2n - (s_0k^3 - (3s_0-2)k^2 + (2s_0-5)k + 6)\right) \\
			\ge{}& 0.
		\end{align*}
		
		Then, $e(G) \le e(G_1) \le e(G_2) \le \binom{n-s_0(k-2)-2}{2} + s_0((k-2)s_0+2)$, contradicting the assumption that $	e(G) > \binom{n-s_0(k-2)-2}{2} + s_0(s_0(k-2)+2)$.
		
		This completes the proof of Theorem \ref{Theorem 1.4}.
	\end{proof}

		\section*{ Declaration of competing interest }
		There is no conflict of interest.

		\section*{ Data availability }
		No data was used for the research described in the paper.
		
		\section*{ Acknowledgements }
		The second author would like to thank the hospitality of the National Institute of Education, Nanyang Technological University in Singapore, where part of the work was done. This work was partly supported by the National Natural Science Foundation of China (No.~12101126), Natural Science Foundation of Fujian Province (No.~2023J01539). This work was also partly supported by China Scholarship Council (No.~202409100010).

	\end{CJK}

\begin{thebibliography}{00}
			\bibitem{a27} G.Y. Ao, R.F. Liu, J.J. Yuan, Sufficient conditions for $k$-factor-critical graphs and spanning $k$-trees of graphs, J. Algebr. Comb. 61 (2025) 28.
			\bibitem{a20} D. Bauer, H.J. Broersma, H.J. Veldman, Not every 2-tough graph is Hamiltonian, Discrete Appl. Math. 99 (2000) 317--321.
			\bibitem{a21} H.J. Broersma, V. Patel, A. Pyatkin, On toughness and hamiltonicity of $2K_2$-free graphs, J. Graph Theory 75 (2014) 244--255.
	    	\bibitem{a39} H.Z. Chen, J.X. Li, S.J. Xu, Some results on the existence of spanning $k$-ended trees in a $t$-connected Graph, Acta Math. Appl. Sin., Engl. Ser. (2025) 1--12.
	    	\bibitem{a1} H.Z. Chen, J.X. Li, S.J. Xu, Two variants of toughness of a graph and its eigenvalues, Graphs Combin. 41 (2025) 41.
	    	\bibitem{a30} H.Z. Chen, X. Lv, J. Li, S.J. Xu, Sufficient conditions for spanning trees with constrained leaf distance in a graph, Discuss. Math. Graph Theory 45 (2025) 253--266.
	    	\bibitem{a25} V. Chv\'{a}tal, Tough graphs and Hamiltonian circuits, Discrete Math. 5 (1973) 215--228.
			\bibitem{a40} K. Das, Maximizing the sum of the squares of the degrees of a graph, Discrete Math. 285 (2004) 57--66.
			\bibitem{a37} D.D. Fan, S. Goryainov, X.Y. Huang, H.Q. Lin, The spanning $k$-trees, perfect matchings and spectral radius of graphs, Linear Multilinear Algebra. 70 (2022) 7264--7275.
			\bibitem{a22} Y. Gao, S. Shan, Hamiltonian cycles in $7$-tough $(P_3 \cup 2P_1)$-free graphs, Discrete Math. 345 (2022) 113069.
			\bibitem{a36} Y. Hong, A bound on the spectral radius of graphs, Linear Algebra Appl. 108 (1988) 135--139.
			\bibitem{a24}  C.L. Jia, Y. Lu, Sufficient conditions for spanning $k$-trees in tough graphs, arXiv:2604.27908 (2026).
			\bibitem{a23} A. Kabela, T. Kaiser, 10-tough chordal graphs are Hamiltonian, J. Comb. Theory, Ser. B 122 (2017) 417--427.
			\bibitem{a33} D. Kratsch, J. Lehel, H. M\"{u}ller, Toughness, hamiltonicity and split graphs, Discrete Math. 150 (1996) 231--246.
			\bibitem{a3} Q. Li, K.Q. Feng, On the largest eigenvalue of a graph, Acta Math. Appl. Sinica. 2 (1979) 167--175.
			\bibitem{a34}  R.F. Liu, A. Fan, J.L. Shu, Spectral extremal problems on factors in tough graphs, and beyond, Discrete Math. 348 (2025) 114593.
			\bibitem{a14} R.F. Liu, W.C. Shiu, J. Xue, Sufficient spectral conditions on Hamiltonian and traceable graphs, Linear Algebra Appl. 467 (2015) 254--266.
			\bibitem{a31} B. Ning, J. Ge, Spectral radius and Hamiltonian properties of graphs, Linear Multilinear Algebra. 63 (2015) 1520--1530.
			\bibitem{a5} K. Ota, M. Sanka, Hamiltonian cycles in 2-tough $2K_2$-free graphs, J. Graph Theory 101 (2022) 769--781.
			\bibitem{a6} S. Shan, Hamiltonian cycles in 3-tough $2K_2$-free graphs, J. Graph Theory 94 (2020) 349--363.
			\bibitem{a8} Y. Shen, L.H. You, M.J. Zhang, S.C. Li, On a conjecture for the signless Laplacian spectral radius of cacti with given matching number, Linear Multilinear Algebra. 65 (2017) 457--474.
			\bibitem{a2} S. Win, On a connection between the existence of $k$-trees and the toughness of a graph, Graphs Combin. 5 (1989) 201--205.
			\bibitem{a26} J. Wu, Characterizing spanning trees via the size or the spectral radius of graphs, Aequat. Math. 98 (2024) 1441--1455.
			\bibitem{a4} S.Z. Zhou, J.C. Wu, Spanning $k$-trees and distance spectral radius in graphs, J. Supercomput. 80 (2024) 23357--23366.
			\bibitem{a28} S.Z. Zhou, Y.L. Zhang, H.X. Liu, Spanning $k$-trees and distance signless Laplacian spectral radius of graphs, Discrete Appl. Math. 358 (2024) 358--365.
			
		
			
			
			
		\end{thebibliography}
\end{document}